\documentclass[12pt,reqno]{amsart}
\usepackage{graphicx} 
\usepackage{enumitem,stix}
\usepackage{epsf}
\usepackage{arydshln}
\usepackage[pagebackref]{hyperref}
\usepackage[margin=1.2in]{geometry}
\usepackage{mathtools}
\usepackage{ytableau}
\usepackage{todonotes}
\usepackage{verbatim}
\usepackage{amsmath}
\usepackage{amsopn}
\usepackage{latexsym}
\usepackage{amsfonts}
\usepackage{amsthm}
\usepackage{tikz}

\newtheorem{thm}{Theorem}[section]
\newtheorem{prop}[thm]{Proposition}

\newtheorem{cor}[thm]{Corollary}

\theoremstyle{definition}

\newtheorem{example}[thm]{Example}

\newtheorem{definition}[thm]{Definition}
\DeclareMathOperator{\des}{des}

\title{The Range of The Des Statistic for Conjugacy Classes in $S_n$}
\author{Yisca Kares }
\thanks{Department of Applied Mathematics, Jerusalem College of Technology, Jerusalem, Israel \\
Michlalah-Jerusalem College, Department of Mathematics, Jerusalem, Israel\\
Department of Mathematics, Open University, Ra'anana, Israel}

\date{}

\begin{document}

\maketitle

\begin{abstract}
We determine the range of the des statistic on every conjugacy class in the symmetric group $S_n$, prove that the minimum is $1$ (except for the identity class), and show that every intermediate value from $1$ to the maximum value is attained.

We also demonstrate a constructive method to achieve every value in the range and discuss its combinatorial implications.
\end{abstract}

\section{Introduction}
Let $S_n$ be the symmetric group on $n$ elements, and let $\pi=[\pi(1) \pi(2) ... \pi(n)]$ be a permutation in $S_n$.

 We say that $i \in \{1,\dots,n-1\}$ is a $descent$ of a permutation $\pi \in S_n$ if $\pi(i)>\pi(i+1)$. 
 Denote $Des(\pi)$ the set of the descents of $\pi$ and $des(\pi)=| Des(\pi)|$ as the number of the descent of $\pi$. 

The $des$ statistic discussed in this work has been extensively studied on the entire symmetric group. Our goal is to restrict attention to individual conjugacy classes and analyze more closely the behavior of the statistic within them, determining both the minimum value and the full range of attained values. 

One of the first to define and study this permutation statistic was Euler.

The distribution function of the des statistic, denoted \(A_n(q)\), is known as the Eulerian p olynomial and is defined via the generating function
\[
\sum_{n \geq 0} A_n(q) \frac{x^n}{n!} = \frac{q-1}{q - e^{(q-1)x}}.
\]

Stieltjes \cite{Stieltjes} established a connection between Eulerian polynomials and continued fractions, given by:
\[
\sum_{n \geq 0} A_n(q) z^n = \frac{1}{1 - \frac{z}{1 - \frac{qz}{1 - \frac{2z}{1 - \frac{2qz}{1 - \cdots}}}}}.
\]
Another property associated with Eulerian polynomials is the property due to Rolle \cite{Rolle}, which states that all roots of the Eulerian polynomial are real, negative, and pairwise distinct.

Euler also introduced the Eulerian numbers \(A_{n,k}\), defined as follows:
\[
A_{n,k} := | \{ \pi \in S_n \mid \text{des}(\pi) = k \} |
\]

Euler \cite{Euler} showed that
$
A_n(q) = \sum_{k \geq 0} A_{n,k} q^k 
$.

Alternatively, using the recurrence relation $ A_{n,k} = (n-k) A_{n-1,k-1} + (k+1) A_{n-1,k} $, shown in 
\cite{Petersen}, one can derive the identity \(A_n(q) = \sum_{k \geq 0} A_{n,k} q^k\). From this, it follows that for all \(n, k\) with \(0 \leq k \leq n\), we have \(0 < A_{n,k}\).
Thus, for all \(n, k\), the coefficient of \(q^k\) in the Eulerian polynomial \(A_n(q)\) is positive. This ensures that every value \(0 \leq k \leq n\) appears as the descent number of some permutation in \(S_n\).
This property guarantees that all values between minimum and maximum are realized, is the continuity property of the $des$ statistic in the symmetric group \(S_n\).

Cellini \cite{Cellini} derived an explicit formula for the Eulerian numbers.

The distributions of statistics on fixed conjugacy classes of \(S_n\) are known
exactly for some classical statistics:
Gessel and Reutenauer \cite{GR} derived a generating function for the joint distribution of descents and the major index over the conjugacy class. 
Brenti \cite{Brenti} provided a generating function for the excedance statistic on conjugacy classes in terms of Eulerian polynomials. 
Several asymptotic results have been established. Fulman \cite{Ful} proved that the distributions of descents and the major index converge to a normal distribution in conjugacy classes with sufficiently large cycles. Kim and Lee \cite{Kim}
later extended this result to all conjugacy class of \(S_n\). 
Loth, Levet, Liu, Stucky, Sundaram, and Yin \cite{LLS} studied the expectations of various permutation statistics, treating them as random variables.

The paper is organized as follows: Section 2 establishes the minimal value attained in each conjugacy class. In Section 3, we prove that in every conjugacy class, except for the class of $\pi=e$, each value from 1 to the maximum value of the des statistic in that class is attained.

If ${C}_n= \{e \}$, the result holds trivially, since $ \{ \des (\pi )|\pi\in {C}_n \} = \{0 \}$. Therefore, throughout the proofs, we assume that ${C}_n\ne \{e \}$.

\section{Minimum value}

In this section, we establish the following result.

\begin{thm}\label{1des}
Let ${C}_n\ne \{ e \}$ be a conjugacy class in the symmetric group $S_{n}$. Then 
\[
\min_{\pi\in C_{n}} des(\pi)=1
\]
\end{thm}

\begin{definition}
Let $n\in \mathbb{N}$, and let $\lambda = ( a_1,...,a_t )\vdash n$, that is, $\lambda$ is a partition of $n$, i.e., $a_1+...+a_t=n$ and $0< a_{i+1}\le a_i$ for all $i\in [1,t-1]$. The {\em Young diagram} $Y_\lambda$ corresponding to the partition $\lambda$, consists of $n$ cells in the plane, arranged in left-justified rows, where the $i$-th row has $a_i$ cells.

\end{definition}

\begin{proof}[proof of theorem \ref{1des}]
The identity permutation $e$ admits a singleton conjugacy class, and it is the only permutation with $des(\pi)=0$. We consider conjugacy classes ${C}_n\ne \{ e \}$.

Let $C_n\ne \{ e \}$. We now show that there exists a permutation $\tau\in C_{n}$ such that $des (\tau)=1$.

Conjugacy classes in ${S}_n$ are determined by their cycle structures, which are partitions of n. Let $ (a_{1},...,a_{k} )$ be the cycle structure of ${C}_n$ such that $0<a_{i+1}\le a_{i}$ for all $1\le i<k$.

Let ${Y}_\lambda$ be the Young diagram corresponding to the cycle structure $\lambda= (a_{1},...,a_{k} )\vdash n$. We fill the boxes of the diagram as follows: the numbers $1,...,n$ will be filled from the leftmost column to the rightmost column, with each column filled from bottom to top in order.

We now write each row in this diagram as a cycle, and let $\tau^{-1}$ be the permutation formed by composing these cycles in reverse order.

We claim that $des (\tau )=1$:


By the construction, we observe that for all $1\le i\le n$ if $i$ does not appear in the first column or the first row of the diagram, then $\tau (i )<\tau (i+1 )$. 

The value ${k}$, which is the number of cycles in $\tau$, appears in the first position in the first column (see example \ref{exa} below). For every ${i}\ne{k}$ appearing in the first column, the value $\tau(i+1)$ appears in the row above the value $\tau(i)$, and in the same column or to the right. Therefore, in this case also, $\tau (i )<\tau (i+1 )$. For every ${i}\ne{k}$ appearing in the first row, the value $\tau(i+1)$ is in a column to the right of the value $\tau(i)$, so again $\tau (i )<\tau (i+1 )$.

Because ${C}_n\ne {e}$, the Young diagram has at least two columns, i.e., ${k}<{n}$ hence we get $\tau (k )=n>\tau (k+1 )$.

Therefore, $des (\tau)=1$, and so $\min_{\pi\in C} des(\pi)=1$.
\end{proof}

\begin{example}\label{exa}

let ${C}_n$ with a cycle structure $ (4,4,3 )\vdash11$.

 The filled diagram will be as follows
 \begin{center}
 \begin{minipage}{0.25\textwidth}
\centering
\begin{ytableau}
 {3} & {6} & {9} & {11} \\
 {2} & {5} & {8} & {10} \\
 {1} & {4} & {7} \\
 \end{ytableau}
 \end{minipage}
 \hspace{0.0005cm}
 \begin{minipage}{0.7\textwidth}
\raggedright
$\Longrightarrow$ $\tau^{-1}= (3,6,9,11 ) (2,5,8,10 ) (1,4,7 )$
 \end{minipage}
 \end{center}
 Hence, the resulting permutation is $\tau= (7,4,1 ) (10,8,5,2 ) (11,9,6,3 )$.
 \end{example}

Let us now recall a result dealing with the number of permutations with minimal $\mbox{des}$ value in the conjugacy class of large cycles:

Towards the end of the paper of Crochemore-Désarménien-Perrin \cite{CDP}, the authors conclude, utilizing the work of Gessel-Reutenauer, that for the conjugacy class of large cycle of length $n$, denoted by $C(n)$, the number of permutations with minimal $\mbox{des}$ value is equal to the number of primitive binary words of length $n$, known as Lyndon words. Explicitly, we have $\# \{ \pi\in C (n )|~des(\pi)=1 \} = L_{n}=\frac{1}{n}\sum_{d|n}\mu \left(\frac{n}{d} \right)2^{d}$, where $\mu$ is the Möbius function. Theorem \ref{1des} shows that the value $L_{n}$ is positive.

\section {Continuity of the des statistic}

We now prove that for each conjugacy class ${C}_n$, every value from 1 to the maximum value is attained.

\begin{definition}
Let ${S}_n$ be the symmetric group and let $1\le i\le n-1$. The i-th {\em Coxeter generator} of ${S}_n$ isthe permutation $s_i=(i,i+1)$.

\end{definition} 

\begin{prop}\label{d=012}
Let $\pi\in S_{n}, 1\le i \le n$, and set $D:= \{ i,i+1 \} \cap \{ \pi (i ),\pi (i+1 ) \} $. Let $\pi':=s_{i}\pi s_{i}^{-1}$. Then, in the transition from $\pi$ to $\pi'$ the following holds:
\begin{itemize}
 
 \item[Case a:] If $|D|={0}$, then the numbers in location $i,i+1$ are swapped, and the numbers $i,i+1$ themselves are swapped.
 
 \item[Case b:] If $|D|={1}$, then the three elements of the set $ \{\pi(i),i,\pi(i+1),i+1 \}$ are cyclically permuted.
 
 \item[Case c:] If $|D|={2}$, then $\pi=\pi'$.
\end{itemize}
 
\end{prop}

\begin{example}
For $i={2}$:
\begin{itemize}
 
 \item[(a)] For $\pi= [\underline{2}\mathbf{61}5\underline{3}4 ]$: Indeed, $ | \{ \pi (2 ),\pi (3 ) \} \cap \{ 2,3 \} |=0$, i.e., $|D|={0}$. We have $\pi s_{2}^{-1}=\pi s_{2}= [2\mathbf{16}534 ]$ and then $s_{2}\cdot\pi s_{2}^{-1}= [\underline{3}\mathbf{16}5\underline{2}4 ]$. The numbers 2,3 have swapped places, and the numbers at positions 2 and 3 have also been swapped.

 \item[(b)] For $\pi= [4\mathbf{52}6\mathbf{3}1 ]$: Indeed, $ | \{ \pi (2 ),\pi (3 ) \} \cap \{ 2,3 \} |= | \{ 2 \} |=1$, i.e., $|D|={1}$. We have $\pi s_{2}^{-1}=\pi s_{2}= [4\mathbf{25}631 ]$ and then $s_{2}\cdot\pi s_{2}^{-1}= [4\mathbf{35}6\mathbf{2}1 ]$. The numbers 5,2,3 are cyclically permuted.

 \item[(c)] For $\pi= [4\mathbf{32}615 ]$: Indeed, $ | \{ \pi (2 ),\pi (3 ) \} \cap \{ 2,3 \} |= | \{ 2,3 \} |=2$, i.e., $|D|={2}$. We have $\pi s_{2}^{-1}=\pi s_{2}= [4\mathbf{23}615 ]$ and then $s_{2}\cdot\pi s_{2}^{-1}= [4\mathbf{32}615 ]$, meaning 2 and 3 have returned to their original positions, then $\pi=\pi'$.
\end{itemize} 
\end{example}

\begin{proof}[proof of proposition \ref{d=012}]
Note that for any permutation $\pi$, the permutation obtained after the composition $\pi {s}_{i}$ swaps the numbers in positions $i$ and $i+1$, while all other values of the permutation remain as in $\pi$. The permutation obtained after the composition ${s}_{i} \pi$ swaps the numbers $i,i+1$, while the other elements remain unchanged.

It is easy to see that for any $t\notin \{ i,i+1,\pi (i ), \pi (i+1 ) \} $ we have $\pi (t )=\pi' (t )$. That is, the permutation does not change after conjugating by ${s}_i$, excluding these elements. 

\begin{itemize}
 \item[(a)] 

 Since $ |D |=0$, in the composition $\pi s_{i}^{-1}$, only the values $\pi(i)$ and $\pi(i+1)$ swap places, and thus in the conjugation $s_{i}\cdot\pi s_{i}^{-1}$, only the values $i$ and $i+1$ swap places from $\pi s_{i}^{-1}$ to $s_{i}\cdot\pi s_{i}^{-1}$.

 \item[(b)] Assume $\pi(i)=i$ (the proofs of the other cases are similar). Then the composition $\pi s_{i}^{-1}=\pi s_{i}$ swaps the value in positions $i$ and $i+1$, meaning it swaps $\pi (i )=i$ and $\pi (i+1 )$, and the number $i+1$ remains in its place. Thus, the composition $s_{i}\cdot\pi s_{i}^{-1}$ swaps the values $i,i+1$ of $\pi s_{i}$. 
 Hence, we obtain $\pi' (i )=\pi (i+1 ),\pi' (i+1 )=i+1,\pi'^{-1} (i )=\pi^{-1} (i+1 )$. In other words, if the order of the three values in the permutation was $\pi= [...,i+1,...,i,\pi (i+1 ),... ]$, then $\pi'= [...,i,...,\pi (i+1 ),i+1,... ]$. That is, the numbers are cyclically permuted, as required. 
 
 \item[(c)] Since $|D|={2}$, the values in positions $i,i+1$ are the numbers $i,i+1$. Therefore, the composition $\pi s_{i}^{-1}=\pi s_{i}$ swaps the values in positions $i,i+1$, and then in the composition $s_{i}\cdot\pi s_{i}^{-1}$ they return to their original places, meaning $\pi'=\pi$, as required.
 
\end{itemize}

\end{proof}

\begin{thm}\label{bigthm}
Let ${C}_n$ be a conjugacy class of ${S}_n$, and denote $M:=\max_{\pi\in {C}_n} \des(\pi)$. Then $ \{ \des (\pi )|\pi\in {C}_n \} = \{1,...,M \}$.
 
\end{thm}

The proof proceeds as follows:
Let $s_{i}= (i,i+1 )$ be the ${i}$-th Coxeter generator,
we get $s_{i}^{-1}=s_{i}$.

There is a permutation $\pi_{0}\in{C}_n$ such that $\text {des}(\pi_{0})=1$ by \ref{1des}, and let $\sigma\in{C}_n$ such that $\text{des}(\sigma)=M$. By the definition of conjugacy classes, these permutations are conjugates, meaning, there exists a permutation $\alpha\in S_{n}$ such that $\sigma=\alpha\pi_{0}\alpha^{-1}$.

In addition, the Coxeter generator $s_{1},...,s_{n-1}$ generate the entire group $S_n$, so there exist generators $\alpha_{1},...,\alpha_{l}\in \{ s_{1},...,s_{n-1} \} $ (repetitions are allowed), such that $\alpha=\alpha_{l}\cdot...\cdot\alpha_{1}$. Therefore, there exists a path $\pi_{0},\pi_{1},...,\pi_{l}=\sigma$, such that for all $0\le j<l,\,\pi_{j+1}=\alpha_{j+1}\pi_{j}\alpha_{j+1}^{-1}$. 

According to Claim \ref{le2}, which will be shown below, it holds that $|\mbox{des}(\pi_{j+1})-\mbox{des}(\pi_{j})|\le2$ for all $j$. This means that for any two consecutive permutations $\pi_j,\pi_{j+1}$ in the path, the value of $des$ increases or decreases by at most 2.
To finish the proof, it remains to close the gapes of size 2. That is, given $0\le j<l$, let $\sigma_{0},\sigma_{2}\in {C}_n$ such that $\sigma_{2}=\alpha_{j}\sigma_{0}\alpha_{j}^{-1}$ and $\mbox{des} (\sigma_{2} )-\mbox{des} (\sigma_{0} )=2$. We will prove that there exists a permutation $\sigma_{1}\in {C}_n $ such that $\mbox{des} (\sigma_{1} )-\mbox{des} (\sigma_{0} )=1$ and this will close the gaps, thereby proving the theorem.

\subsection{Difference in the path}

\begin{prop}\label{not consecutive} Let $\pi\in {S}_n$ and let $1\le i\le n$. If $\pi^{-1}(i),\pi^{-1}(i+1)$ are not consecutive numbers, then $\mbox{des}(\pi)=\mbox{des}(\pi s_i)$. 
 
\end{prop}

\begin{proof}
$\pi^{-1} (i ),\pi^{-1} (i+1 )$ are not consecutive numbers and then $i,i+1$ are not in consecutive positions in the sequence of values of $\pi$, thus they are not in consecutive positions in the sequence of values of $\pi s_i$. 

Note that for any $a\ne i,i+1$, we have $a<i\iff a<i+1$.
Therefore, $\mbox{Des}(\pi)=\mbox{Des}(\pi s_i)$ because $\pi s_i,\pi$ are identical except for the swap between the values $i,i+1$, and thus we get $\mbox{des}(\pi)=\mbox{des}(\pi s_i)$
\end{proof}

\begin{prop}\label{pis_i} Let $\pi\in {S}_n$ and $1\le i\le n$. then $|\mbox{des}(\pi)-\mbox{des}(s_i\pi )|\le1$. 
 
\end{prop}

\textbf{Definition:}
 Let $a= \langle a_{1},...,a_{m} \rangle ,b= \langle b_{1},...,b_{m} \rangle$ be two sequences of distinct numbers, and let $\varphi:a\to b$ such that $\varphi(a_{i})=b_{i}$. We say that these sequences are of the \textbf{same type} if $\varphi(a_{i})<\varphi(a_{j})\iff a_{i}<a_{j}$ for all $i\ne j$.

\begin{proof}
In the sequence of values of $\pi$, the relevant part is $[...,\pi(i-1),\pi(i),\pi(i+1),\pi(i+2),...]$. If $i=1$ or $i=n-1$, there are only three relevant values here.

It is easy to see that $\mbox{Des} (a )=\mbox{Des} (b )$ for two sequences $a,b$ of the same type, and then $\des(a)=\des(b)$.

Additionally, for every sequence $a$ of length $m$, there exists a unique sequence of the same type $b=\langle b_{1},...,b_{m} \rangle$ of length $m$, such that $\{ b_{1},...,b_{m}\} =\{ 1,\ldots,m\} $.

Hence, we can discuss all the permutations $\sigma\in S_{4}$ instead of discussing the numbers $\pi(i-1),\pi (i ),$ $\pi (i+1 ),\pi(i+2)$ in the permutation $\pi$. 
We will examine each possibility for the swap between $\pi (i+1 )$ and $\pi (i )$ in the transition from $\pi$ to $\pi s_i$, and how it affect their $\des$ value: 

Notice that $S_{4}$ decomposes into 12 pairs of permutations $\{ (a,b,c,d),(a,c,b,d)\} $ and we are looking at the difference between their $\des$ values for each pair. Additionally, for every permutation $[a,b,c,d]$ we get $\des([a,b,c,d])+\des([d,c,b,a])=3$. Therefore, we can assume \(a < d\), and it suffices to go through 6 pairs: 

\begin{table}[h!]\label{des table}
\centering
\begin{tabular}{|c|c|c|c|c|c|c|}
\hline
$\pi$ & 1234 & 1243 & 1342 & 2143 & 2134 & 3214 \\
\hdashline
$\mbox{des}(\pi)$ & 0 & 1 & 1 & 2 & 1 & 2\\
\hline
$s_i\pi$ & 1324 & 1423 & 1432 & 2413 & 2314 & 3124 \\
\hdashline
$\mbox{des}(s_i\pi)$ & 1 & 1 & 2 & 1 & 1 & 1\\
\hline
\end{tabular}
\caption{6 pairs and their des value }
\label{table:simple}
\end{table}

From the table, it can be seen that $|\mbox{des}(\pi)-\mbox{des}( s_i\pi)|\le1$
\end{proof}

\begin{thm}\label{les equal 1} Let $\pi\in {C}_n$ and let $1\le i\le n$. If case (b) or (c) of Proposition \ref{d=012} holds, or if case (a) holds such that $\pi^{-1} (i ),\pi^{-1} (i+1 )$ are not consecutive numbers, then \[\mid {\mbox{des}(\pi)-\mbox{des}(s_{i}^{-1}\pi s_{i})}\mid \le1 
 \]

\end{thm}

\begin {proof}
Let $\pi''=s_{i}\pi$ and $\pi'=\pi'' s_{i}$, which means $\pi'=s_{i}^{-1}\pi s_{i}$.
$\pi''$ is the permutation obtained from $\pi$ by swapping the values $i,i+1$.

\begin{itemize}
\item[(A)] If $\pi$ satisfies case (a) such that $\pi^{-1} (i ),\pi^{-1} (i+1 )$ are not consecutive numbers:
From $\pi''=s_{i}\pi$ we get $\mbox{des} (\pi )=\mbox{des} (\pi'' )$ by \ref{not consecutive}. In \ref{pis_i} we prove that $ |\mbox{des} (\pi )-\mbox{des} (\pi s_{i} ) |\le1$

Since \(|D| = 0\), we get $ |\mbox{des} (\pi )-\mbox{des} (s_{i}^{-1}\pi s_{i} ) |\le1$, for all $1<i<n-1$. 

If $i=1$ or $i=n-1$, it can be seen similarly that the inequality holds.

\item[(B)] If $\pi$ satisfies case (b):
Suppose that $\pi (i )=i$ (the other cases hold similarly).

$\pi= [...i...i+1... ]$ then $\pi''= [...i+1...i... ]$.

There are two cases:

1. $\pi^{-1} (i ),\pi^{-1} (i+1 )$ are not in consecutive positions, then from \ref{not consecutive} we get $\mbox{des} (\pi )=\mbox{des} (\pi'' )$, and from \ref{pis_i} we get $ |\mbox{des} (\pi )-\mbox{des} (s_{i}^{-1}\pi s_{i} ) |\le1$.

2. $\pi^{-1} (i ),\pi^{-1} (i+1 )$ are in consecutive position, then the permutation $\pi$ looks like:
\[ \pi= \left (\begin{array}{ccccc}
... & i-1 & i & i+1 & ...\\
... & i+1 & i & \pi (i+1 ) & ...
\end{array} \right) \] (recall that $\pi (i )=i$).

Then $\pi''=s_{i}\pi= \left(\begin{array}{ccccc}
... & i-1 & i & i+1 & ...\\
... & \mathbf{i} & \mathbf{i+1} & \pi (i+1 ) & ...
\end{array} \right)$ and then \[\pi'=\pi''s_{i}=\left (\begin{array}{ccccc}
... & i-1 & i & i+1 & ...\\
... & i & \pi (i+1 ) & i+1 & ...
\end{array}\right ) \] 
We note that $j\in\mbox{Des} (\pi' )\iff j\in\mbox{Des} (\pi )$ for every $j\notin \{ i-2,i-1,i,i+1 \} $.

For $j=i-2$, note that $\pi (i-2 )< i+1 \iff\pi (i-2 )<i$. Thus, we get $i-2\in\mbox{Des} (\pi )\iff i-2\in\mbox{Des} (\pi' )$.

It remains to check what happens for the positions $ \{ i-1,i,i+1 \} $. We distinguish between two cases:

(1) If $i<\pi (i+1 )$, then $i-1\notin\mbox{Des} (\pi' ),i-1\in\mbox{Des} (\pi )$ and $i\notin\mbox{Des} (\pi )$. Additionally, since $i<\pi (i+1 )$, we also have $i+1<\pi (i+1 )$, hence $i\in\mbox{Des} (\pi' )$. Therefore, except perhaps for $i+1$, there are the same number of elements in the sets $\mbox{Des} (\pi' )$ and $\mbox{Des} (\pi )$, thus $ |\mbox{des} (\pi )-\mbox{des} (s_{i}^{-1}\pi s_{i} ) |\le1$.

(2) If $\pi (i+1 )<i$, then $i-1\in\mbox{Des} (\pi' ),i-1,i\in\mbox{Des} (\pi )$. Since $\pi (i+1 )<i$, we have $\pi (i+1 )<i+1$, hence $i\notin\mbox{Des} (\pi' )$. 

Consider $i+1$:

If $i+1\in\mbox{Des} (\pi )$, it means $\pi (i+2 )<\pi (i+1 )$. Since $\pi (i+1 )<i+1$, we also have $\pi (i+2 )<i+1$, i.e., $i+1\in\mbox{Des} (\pi' )$.

If $i+1\notin\mbox{Des} (\pi )$ then $i-1,i\in\mbox{Des} (\pi ),i+1\notin\mbox{Des} (\pi )$, $i-1\in\mbox{Des} (\pi' )$ 
 and $i\notin\mbox{Des} (\pi' )$.

Hence, whether $i+1\in\mbox{Des} (\pi' )$ or not, the values $\mbox{des} (\pi )$ and $\mbox{des} (\pi' )$ differ by at most 1.

\item[(C)] If Case (c) holds, $\pi=\pi'$ so $\mbox{des} (\pi )=\mbox{des} (s_{i}^{-1}\pi s_{i} )$.

\end{itemize} 

Therefore, for all cases: $ |\mbox{des} (\pi )-\mbox{des} (s_{i}^{-1}\pi s_{i} ) |\le1$.

\end {proof}

In this theorem, we generalize what happens for the difference $ |\mbox{des} (\pi )-\mbox{des} (s_{i}^{-1}\pi s_{i} ) |$ for all cases .

\begin{cor}\label{le2}
Let ${C}_n$ be a conjugacy class of ${S}_n$, and let $\pi\in {C}_n$,$1\le i\le n$. 

Then $ |\mbox{des} (\pi )-\mbox{des} (s_{i}^{-1}\pi s_{i} ) |\le2$ .
 
\end{cor}

\begin{proof}
In Claim \ref{les equal 1} we proved that $ |\mbox{des}(\pi)-\mbox{des} (s_{i}^{-1}\pi s_{i} ) |\le1$ for some cases.

It remains to show about case (a) such that $\pi^{-1} (i ),\pi^{-1} (i+1 )$ are consecutive:

Set $m_i= \min \{ \pi^{-1} (i ),\pi^{-1} (i+1 ) \}$. Then $m_i$ is either added or removed from $\mbox{Des}(\pi)$ to $\mbox{Des}(s_i \pi)$.

Furthermore, there is at most one change from $\mbox{Des} (s_{i}\pi )$ to $\mbox{Des} (s_{i}\pi s_{i}^{-1} )$, as shown in the table in the proof of Claim \ref{pis_i}.

Therefore, in total, there are at most 2 changes in $\mbox{des}(\pi)$ compared to $\mbox{des}(s_i^{-1} \pi s_i)$, that is, $ |\mbox{des} (\pi_{j} )-\mbox{des} (\pi_{j+1} ) |\le2$ 
\end{proof}

This Corollary implies that $ |\mbox{des} (\pi_{j} )-\mbox{des} (\pi_{j+1} ) |\le2$ at each step in the sequence $\pi_{0},...,\pi_{l}$.

\begin{cor}\label{cor i}
 If $\mbox{des} (s_{i}^{-1}\pi s_{i} )-\mbox{des} (\pi )=2$ then $\pi^{-1}(i)<\pi^{-1}(i+1)$ and $\pi(i)<\pi(i+1)$.
\end{cor}

\subsection{Preparation for the main result}

We now aim to close the gaps in the cases where the difference is greater than 1 (i.e., in case (a), where $\pi^{-1} (i ),\pi^{-1} (i+1 )$ are consecutive).

Let $w= (1,...,n )$, then
$w^{-1}= (n,...,1 )$.

\begin{definition}
Let $\pi \in {S}_n$, and denote $\pi (n+1 ):=\pi (1 )$. The {\em cyclic descent set} of $\pi$ is $\mbox{cDes}(\pi) :=\{i \in \{1,\dots,n\} \mid \pi(i)>\pi(i+1)\}$. 

\end{definition}

\begin{definition}
Let $\pi \in {S}_n$. The {\em cyclic descent number} of $\pi$ is $\mbox{cdes}(\pi):= |\mbox{cDes} (\pi ) |$. 
 
\end{definition}

\begin{prop} \label{cdes}

Let $\pi\in {S}_n$, and $w=(1,...,n)$, then $\mbox{cdes} (w^{-k}\pi w^{k} )=\mbox{cdes} (\pi )$ for all $k\in\mathbb{Z}$.
 
\end{prop}

\begin{example} For $\pi= [3142 ]$.
$\mbox{cdes} (\pi )=2$. We have:

$w^{-1}\pi w= [4312 ] \Rightarrow \mbox{cdes} (w^{-1}\pi w^{1} )=2$.

$ w^{-2}\pi w^{2}= [2413 ] \Rightarrow \mbox{cdes} (w^{-2}\pi w^{2} )=2$.

$w^{-3}\pi w^{3}= [3421 ] \Rightarrow \mbox{cdes} (w^{-3}\pi w^{3} )=2$. 

$w^{4}=e$.
Therefore, the proposition holds for all $k\in\mathbb{Z}$.

\end{example}

\begin{proof}

Let $\pi= [\pi (1 ),...,\pi (n ) ]$.

It holds that $w\pi w^{-1}= [\pi (n )+1,\pi (1 )+1,...,\pi (n-1 )+1 ] (\mbox{mod}(n))$, meaning each value $\pi (i )$ shifts cyclically to the right and increases by 1 modulo $n$.

Therefore, except for the positions $\pi^{-1}(n) - 1, \pi^{-1}(n)$ in $\pi$, which moved to the positions $\pi^{-1}(n), \pi^{-1}(n) + 1$ in $w \pi w^{-1}$, there is no change in the value of $cdes$, because every number is increases by 1 (which does not change the order relation between the elements in the sequence), and the cyclic shift to the right also does not change the value, except for the aforementioned places.

In the places where the change occurred:

$\pi^{-1} (n )\in\mbox{cDes} (\pi )$ and $\pi^{-1} (n )-1\notin\mbox{cDes} (\pi )$, because $n$ is the largest number among $1,...,n$. 

The value $n$ shifts to the right, and increases by 1 mod $n$, meaning it shifted to the right and became the value $1$, which is the smallest number among $1,...,n$. therefore, we obtain that $\pi^{-1} (n )\in\mbox{cDes} (w\pi w^{-1} )$ and $\pi^{-1} (n )+1\notin\mbox{cDes} (w\pi w^{-1} )$.

Thus, $\mbox{cdes} (w\pi w^{-1} )=\mbox{cdes} (\pi )$ and hence $\mbox{cdes} (w^{-k}\pi w^{k} )=\mbox{cdes} (\pi )$ for all $k\in\mathbb{Z}$.
 
\end{proof}

\begin{prop}\label{cdes12}
Let $\pi$ be a permutation satisfying case (a) of Claim \ref{les equal 1} such that the numbers $\pi^{-1}(i),\pi^{-1}(i+1)$ are consecutive numbers, and let $k:=\min \{ \pi^{-1}(i),\pi^{-1}(i+1) \} $. Then:

(1) if $\pi^{-1}(i)<\pi^{-1}(i+1)$, then $\mbox{cdes} (w^{-k}\pi w^{k} )=\mbox{des} (w^{-k}\pi w^{k} )$.

(2) if $\pi^{-1}(i)>\pi^{-1}(i+1)$, then $\mbox{cdes} (w^{-k}\pi w^{k} )=\mbox{des} (w^{-k}\pi w^{k} )+1$.
 
\end{prop}

We can note that $w^{-1}\pi w= [\pi (2 )-1,...,\pi (n )-1,\pi (1 )-1 ](\mbox{mod} (n))$.

For example, let $i=5,\,\pi= [246531 ]$. In this case, $\pi^{-1}(5),\pi^{-1}(6)$ are consecutive numbers, and $ |D |=0$. 

We have $k=\min \{ \pi^{-1} (5 ),\pi^{-1} (6 ) \} =\min \{ 4,3 \} =3$.

For this permutation, $\pi^{-1}(5)>\pi^{-1}(6)$ (which corresponds to case (2)), we calculate the permutations $\mbox{cdes} (w^{-3}\pi w^{3} ),\mbox{des} (w^{-3}\pi w^{3} )$:

We find $w^{-3}\pi w^{3}= [264513 ]$. Thus, $\mbox{cdes} (w^{-3}\pi w^{3} )=3,\mbox{des} (w^{-3}\pi w^{3} )=2$, and then $\mbox{cdes} (w^{-3}\pi w^{3} )=\mbox{des} (w^{-3}\pi w^{3} )+1$. 

Note that the equality $\mbox{cdes} (w^{-k}\pi w^{k} )=\mbox{des} (w^{-k}\pi w^{k} )+1$ does not hold for every $k$. For example, when $k=4$, we get $\mbox{cdes} (w^{-k}\pi w^{k} )=\mbox{des} (w^{-k}\pi w^{k} )$.

\begin{proof}
We have $w^{-1}\pi w= [\pi (2 )-1,...,\pi (n )-1,\pi (1 )-1 ](\mbox{mod}(n))$.

Then $ (w^{-1}\pi w ) (t )=\pi (t+1 )-1$ for $1\le t <n$, and $ (w^{-1}\pi w ) (n )=\pi(1)-1\,(\mbox{mod}(n))$

Therefore, we get $ (w^{-k}\pi w^{k} ) (n )=\pi (k )-k,\,\,\,\, (w^{-k}\pi w^{k} ) (1 )=\pi(k+1)-k\,(\mbox{mod}(n))$.

We get $\pi^{-1}(i)<\pi^{-1}(i+1)\Longleftrightarrow k=\pi^{-1} (i )\Longleftrightarrow \pi (k )=i\Longleftrightarrow 
 \pi (k )<\pi (k+1 )$ from the definition of k. 

Moreover, $ \{ \pi (k )-k,\pi (k+1 )-k \} = \{ i-k,i+1-k \} $. Since $ \{ i,i+1 \} \cap \{ k,k+1 \} =\emptyset$ (case (a)), the two consecutive numbers $ \{ i-k,i+1-k \} $ are different from $n\, (\mbox{mod}(n) )$. Therefore, $\pi (k )<\pi (k+1 )\Leftrightarrow\pi (k )-k<\pi (k+1 )-k\,(\mbox{mod}(n))$, which implies that $\pi (k )<\pi (k+1 )\Leftrightarrow (w^{-k}\pi w^{k} ) (n )< (w^{-k}\pi w^{k} ) (1 )$. 

Thus, $\pi^{-1}(i)<\pi^{-1}(i+1) \Leftrightarrow\pi (k )<\pi (k+1 ) \Leftrightarrow (w^{-k}\pi w^{k} ) (n )< (w^{-k}\pi w^{k} ) (1 )$, as required.
 
\end{proof}

\subsection{The main result: The continuity}

\begin{prop}\label{wt}
 Let $\sigma= [t+1,...,n,1,...,t-1,t ]$ for $1\le t\le n$. 
 
 Then $\mbox{des} (s_{t}^{-1}\sigma s_{t} )=\mbox{des} (\sigma )+{1}$
\end{prop}

\begin{proof}
First, $\mbox{des} (\sigma )={1}$. We show that $\mbox{des} (s_{t}^{-1}\sigma s_{t} )={2}$.

- If $\sigma (t )=n$: $\sigma (t+1 )=1$ so $s_{t}^{-1}\sigma s_{t}= [t,t+2...,n-1,1,n,2,...,t-1,t+1 ]$, then $\mbox{des} (s_{t}^{-1}\sigma s_{t} )=2$.

-If $\sigma (t )< n$: Since $\sigma (t )<\sigma (t+1 )$ and these numbers are consecutive numbers, the swap of the values $t,t+1$ (the permutation $\sigma''=s_{t}\sigma$) will not change the $\mbox{des}$ values, and the swapping of $\sigma (t ),\sigma (t+1 )$ will increase the $\mbox{des}$ by 1, so $\mbox{des} (s_{t}^{-1}\sigma s_{t} )=2$.

\end{proof}
 
\begin{prop}\label{aa}
 Let $\sigma\in C_n, 1\le i\le n$ such that $\sigma (n )<\sigma (1 )$ and $\mbox{des} (s_i^{-1} \sigma s_i )-\mbox{des} (\sigma )=2$ then $ (s_i^{-1} \sigma s_i )(n)< (s_i^{-1} \sigma s_i )(1)$ 
\end{prop}

\begin{proof}
 denote $\sigma_i=s_i^{-1} \sigma s_i$. From \ref{cor i} we have $\sigma^{-1} (i )<\sigma^{-1} (i+1 )$
 
(1) If $1<i,i+1<n$:

$\sigma (1 ),\sigma (n )$ did not change when transitioning from $\sigma$ to $\sigma_{i}$, unless at least one of them is $i$ or $i+1$. If only one of them is $i$ or $i+1$, then it still holds that $\sigma_{i} (n )<\sigma_{i} (1 )$.

It cannot be that both of them are $i,i+1$ because then $\sigma (n )=i,\sigma (1 )=i+1$ due to $\sigma (n )<\sigma (1 )$. This contradicts our earlier conclusion that $\sigma^{-1} (i )<\sigma^{-1} (i+1 )$. 

(2) If $i+1=n$, i.e., $i=n-1$:


From \ref{cor i} we have $\sigma (n-1 )<\sigma (n )$, therefore, from the given, $\sigma(n)<\sigma(1)$, it follows that $\sigma(n-1)<\sigma(1)$.

$\sigma (n-1 )=\sigma_{i} (n )$ and $\sigma (n )=\sigma_{i} (n-1 )$, then $\sigma_{i} (n )=\sigma (n-1 )<\sigma (1 )$. 

Furthermore, $\sigma_{i}(1)=\sigma(1)$ or $\sigma_{i}(1)=\sigma(1)+1$, so overall $\sigma_{i} (n )<\sigma_{i} (1 )$.

(3) If $i=1$:

$\sigma= [\sigma (1 ),\sigma (2 ),...,1,2,... ] $, then $\sigma_{i} (2 )=\sigma (1 )<\sigma (2 )=\sigma_{i} (1 )$ from the same reasoning as in case (2). 
Hence, $\sigma_{i} (n )\le\sigma (n )<\sigma (1 )=\sigma_{i} (2 )<\sigma_{i} (1 )$ (we get $\sigma^{-1} (1 )<\sigma^{-1} (2 )$ as mentioned above, because $i=1$. If $\sigma (n )\ne2$ then $\sigma_{i} (n )=\sigma (n )$. If $\sigma (n )=2$ then $\sigma_{i} (n )=1<2=\sigma (n )$. Therefore, in total, $\sigma_{i} (n )<\sigma_{i} (1 )$.

\end{proof}

Now we proceed to prove Theorem \ref{bigthm}:

\begin{proof}[proof for \ref{bigthm}]
As stated at the beginning, it remains to show that given $\sigma_{0},\sigma_{2}\in {C}_n$ such that, $\sigma_{2}=s_{i}^{-1}\sigma_{0}s_{i}$, and $ \mbox{des} (\sigma_{2} )-\mbox{des} (\sigma_{0} )=2$ (where $s_{i}= (i,i+1 )$ is a Coxeter generator), there exists a permutation $\sigma_{1}\in {C}_n$ such that $\mbox{des} (\sigma_{1} )-\mbox{des} (\sigma_{0} )=1$, and this will complete the proof. 

According to Claim \ref{les equal 1}, if $\mbox{des} (\sigma_{2} )-\mbox{des} (\sigma_{0} )=2$, then we are in case (a) of \ref{d=012}, where $\sigma_{0}^{-1} (i ),\sigma_{0}^{-1} (i+1 )$ are consecutive. Therefore, we will only address this case. 

If $\sigma_0\ne w^t$ for all $1\le t\le n$, Define $\sigma_{1}$ as one of the conjugations $w^{-k}\sigma_{0}w^{k}$ or $w^{-k}\sigma_{2}w^{k}$, where \(k\) and the choice between the two conjugations will be determined as follows:

According to $ \mbox{des} (\sigma_{2} )-\mbox{des} (\sigma_{0} )=2$, we get from \ref{cor i} that $\sigma_{0}^{-1} (i )<\sigma_{0}^{-1} (i+1 )$. 

We distinguish two cases:

a) $\sigma_{0} (n )>\sigma_{0} (1 )$:

In this case, $\mbox{cdes} (\sigma_{0} )=\mbox{des} (\sigma_{0} )+1$.

Since $\sigma_{0}^{-1} (i )<\sigma_{0}^{-1} (i+1 )$ according to Lemma \ref{cdes12}, for $k=\sigma_{0}^{-1}(i)$ we obtain that $w^{-k}\sigma_{0}w^{k}$ satisfies $\mbox{cdes} (w^{-k}\sigma_{0}w^{k} )=\mbox{des} (w^{-k}\sigma_{0}w^{k} )$. 

Then, we get $\mbox{des} (w^{-k}\sigma_{0}w^{k} )=\mbox{cdes} (w^{-k}\sigma_{0}w^{k} )=\mbox{cdes} (\sigma_{0} )=\mbox{des} (\sigma_{0} )+1$, according to Lemma \ref{cdes}. 

Therefore $\sigma_{1}=w^{-k}\sigma_{0}w^{k}\in C_n$ satisfies the requirement.

b) $\sigma_{0} (n )<\sigma_{0} (1 )$:

In this case, $c\mbox{des} (\sigma_{0} )=\mbox{des} (\sigma_{0} )$. 

In \ref{aa} We get $\sigma_{2} (n )<\sigma_{2} (1 )$, 
Moreover $\sigma_{2}^{-1} (i )>\sigma_{2}^{-1} (i+1 )$ holds, since the values $i,i+1$ were swapped, and $\sigma_{0}^{-1} (i )<\sigma_{0}^{-1} (i+1 )$. Therefore, the permutation $w^{-k}\sigma_{2}w^{k}$  satisfies $\mbox{cdes} (w^{-k}\sigma_{2}w^{k} )=\mbox{des} (w^{-k}\sigma_{2}w^{k} )+1$, where $k=\sigma_{2}^{-1}(i+1)$, according to \ref{cdes12}. 

Hence,  $\mbox{des} (w^{-k}\sigma_{2}w^{k} )=\mbox{cdes} (w^{k}\sigma_{2}w^{-k} )-1=\mbox{cdes} (\sigma_{2} )-1=\mbox{des} (\sigma_{2} )-1=\mbox{des} (\sigma_{0} )+1$, by \ref{cdes}. Thus, $\sigma_{1}=w^{-k}\sigma_{2}w^{k}\in {C}_n$ satisfies the required condition.

If $\sigma_0=w^t$ for some $1\le t\le n$, we can see that $w= [t+1,...n,1,...t-1,t ]$, and then by \ref{wt} we take $\sigma_1=s_{t}^{-1}\sigma s_{t}$ (Note that $\mbox{des} (w^{t} )=1$ for all $1\le t<n$, and then $\sigma_{2}\ne w^{t}$, thus only for $\sigma_{0}$ we assume it is different from \(w^{t}\)).

Overall, we have established that if there is a jump of 2 in the transition from permutation \(\sigma_{0}\) to \(\alpha_{j+1} \sigma_{0} \alpha_{j+1}^{-1}\), then it can be resolved by conjugating \(\sigma_{0}\) or \(\sigma_{2}\) with the permutation \(w^{k}\) for some appropriate \(k \in \mathbb{Z}\) or with $s_i$, and this permutation is \(\sigma_{1}\).

\end{proof}

\section*{Acknowledgments}
This paper is based on a part of the doctoral thesis of the author, supervised by Ron Adin and Yuval Roichman. The author thanks them for stimulating discussions and fruitful ideas.


\begin{thebibliography}{}


\bibitem{Rolle}
M. Bona, 
\emph{Combinatorics of permutations,}
Chapman\&Hall, CRC, 2004.

\bibitem{Brenti}
F. Brenti, \emph{Permutation enumeration, symmetric functions and unimodality,}
Pacific Journal of Mathematics Vol. 157 (1993), No.1. 1-28.


\bibitem{Cellini}
P. Cellini, 
\emph{Cyclic eulerian elements,}
European Journal of Combinatorics 19 (1998), 545-552.

\bibitem{CDP}
M. Crochemore, J. Désarménien, \& D. Perrin, \emph{A note on the Burrows-Wheeler transformation}, Theoretical Computer Science 332 (2004), 567-572.

\bibitem{Euler}
L. Euler,
\emph{Methodus universalis series summandi ulterius promota,} 
Euler Archive – All Works, 55 (1741), 147-158.

\bibitem{Ful}
J. Fulman,
\emph{The Distribution of Descents in Fixed Conjugacy Classes of the Symmetric Groups,}
Journal of Combinatorial Theory, Series A,
Volume 84, Issue 2,
1998,
Pages 171-180

\bibitem{GR}
I. M. Gessel and C. Reutenauer, 
\emph{Counting Permutations with given cycle structure and descent set,}
J. Combinatorial Theory (Ser. A) 64 (1993), 189-215.

\bibitem{Kim}
G. B. Kim and S. Lee, 
\emph{Central limit theorem for descents in conjugacy classes of Sn},
J. Combinatorial Theory (Ser. A), 169:105123, 2020.

\bibitem{LLS}
M. Levet, K. Liu, J. Loth, E. Stucky, S. Sundaram, \& M. Yin, 
\emph{Permutation Statistics in Conjugacy Classes of the Symmetric Group}, (2023). 10.48550/arXiv.2301.00898. 

\bibitem{Petersen}
T. K. Petersen,
\emph{Eulerian numbers,} 
Birkhäuser Advanced Texts, Birkhäuser, 2015.

\bibitem{Stieltjes}
T. J. Stieltjes, 
\emph{ Sur la reduction en fraction continue d'une série procedant suivant les puissances descentantes d'une variable,}
Ann. Fac. Sc. Toulouse 3 (1889), 1-17.



\end{thebibliography}
\end{document}